\documentclass[12pt,a4paper]{article}
\usepackage[top=3cm, bottom=3cm, left=3cm, right=3cm]{geometry}
\usepackage[latin1]{inputenc}
\usepackage{amsmath}
\usepackage{amsthm}
\usepackage{amsfonts}
\usepackage{amssymb}
\usepackage{graphics}
\usepackage[hang,flushmargin]{footmisc} %nonindent footnote

\usepackage{amsmath,amsthm,amssymb,graphicx, multicol, array}
\usepackage{enumerate}
\usepackage{enumitem}
\usepackage{xcolor}
\usepackage{currfile}

\usepackage{tabto}

\usepackage{tikz}
\usepackage{pgfplots}

\newtheorem{thm}{Theorem}[section]
\newtheorem{lem}{Lemma}[section]

\newtheorem{dfn}{Definition}[section]
\newtheorem{exe}{Exercise}[section]

\newcommand{\N}{\mathbb{N}}

\newcommand{\R}{\mathbb{R}}

\usepackage{fancyhdr}
\title{Irrationality Measure Of $\pi^2$}
\date{}
\author{N. A. Carella}

\begin{document}
\maketitle

\textbf{\textit{Abstract}:} The note provides a simple proof of the irrationality measure $\mu(\pi^2)=2$ of the real number $\pi^2$, the same as almost every irrational number. The current estimate gives the upper bound $\mu(\pi^2)\leq 5.0954 \ldots$.

\let\thefootnote\relax\footnote{ \today \date{} 
}

%SSSSSSSSSSSSSSSSSSSSSSSSSSSSSSSSSSSSSSSSSSSSSSSSSSSSSSSSSSSSSSSSSSSSSSSSSSSSSSSSSSSSSSSSSSSSSSSSSSSSSS
%SSSSSSSSSSSSSSSSSSSSSSSSSSSSSSSSSSSSSSSSSSSSSSSSSSSSSSSSSSSSSSSSSSSSSSSSSSSSSSSSSSSSSSSSSSSSSSSSSSSSSS
%SSSSSSSSSSSSSSSSSSSSSSSSSSSSSSSSSSSSSSSSSSSSSSSSSSSSSSSSSSSSSSSSSSSSSSSSSSSSSSSSSSSSSSSSSSSSSSSSSSSSSS
\section{Introduction and the Result} \label{SP5555}
The irrationality measure measures the quality of the rational approximation of an irrational number. 
It is lower bound for all the rational approximations. The concept of measures of irrationality of real numbers is discussed in \cite[p.\ 556]{WM2000}, \cite[Chapter 11]{BB1987}, et alii. This concept can be approached from several points of views. 

\begin{dfn} \label{dfnP5555.101} {\normalfont The irrationality measure $\mu(\alpha)$ of a real number $\alpha \in \R$ is the infimum of the subset of  real numbers $\mu(\alpha)\geq1$ for which the Diophantine inequality
\begin{equation} \label{eqP5555.036}
  \left | \alpha-\frac{p}{q} \right | \ll\frac{1}{q^{\mu(\alpha)} }
\end{equation}
has finitely many rational solutions $p$ and $q$. Equivalently, for any arbitrary small number $\varepsilon >0$
\begin{equation} \label{eqP5555.046}
  \left | \alpha-\frac{p}{q} \right | \gg\frac{1}{q^{\mu(\alpha)+\varepsilon} }
\end{equation}
for all large $q \geq 1$.
}
\end{dfn}

The irrationality measure of the pi power $ \pi^n$ is unknown for every integer $n\geq1$. However, there are many estimates. The special cases $\pi^2$ has been studied by several authors, the current record is $\mu(\pi^2)=5.0954 \ldots$, see \cite{HM1993}, \cite{ZW2013}, and similar references. \\

\begin{thm} \label{thmP5555.200}  The irrationality measure of the irrational number $\pi^2$ is $\mu(\pi^2)=2$.
\end{thm}

\begin{proof} Let $\{p_n/q_n:n\geq1\}$ be the sequence of convergents of the irrational number $\pi^2$, see Section \ref{SP2000}. Now, suppose that 
\begin{equation}\label{eqP5555.050}
\mu=\mu(\pi^2)>2,  
\end{equation}
and consider
\begin{equation} \label{eqP5555.100}
\left |\sin \left (\pi^3 q_n\right) \right |
=\left |\sin \left (\pi^3 q_n-\pi p_n\right)\right |
=\left |\sin \pi \left (\pi^2 q_n- p_n \right )\right |.
\end{equation}
By Theorem \ref{thmP7777.200} the sequence of real number $z_n=\left |\pi^2 q_n-p_n \right | <1$ tends to $1/q_n^{\mu}$ as $n\to \infty$. Substituting it into the inequality 
\begin{equation}\label{eqP5555.110}
|z|\ll |\sin z|\ll |z|
\end{equation}
for $|z|<1$, leads to the symmetric inequality
\begin{equation}\label{eqP5555.150}
\left |\pi^2 q_n- p_n \right |\ll\left |\sin \left (\pi^2 q_n- p_n\right) \right |\ll\left |\pi^2 q_n - p_n\right |
\end{equation}
for all large integers $n\geq1$. The above inequality clearly shows that the lower bound of the sine function is independent of the irrationality measure of the sequence of real numbers $z_n=\left |\pi^2 q_n-p_n \right |$. Furthermore, the Diophantine inequality
\begin{equation}\label{eqP5555.120}
\frac{1}{q_n^{\mu-1}} \ll\left |\pi^2 q_n - p_n\right |\ll\frac{1}{q_n^{2}}, 
\end{equation}
for some $\mu(\pi^2)>2$, and the hypothesis \eqref{eqP5555.050} imply the symmetric inequality
\begin{equation}\label{eqP5555.130}
\left |\pi^2 q_n- p_n \right |\ll\frac{1}{q_n} \ll \left |\sin \left (\pi^2 q_n- p_n\right) \right |\ll\left |\pi^2 q_n - p_n\right |\ll\frac{1}{q_n}  .
\end{equation}

Otherwise,  
\begin{equation}\label{eqP5555.140}
\left |\pi^2 q_n- p_n \right |\ll\frac{1}{q_n^{\mu-1}} \ll \left |\sin \left (\pi^2 q_n- p_n\right) \right |\ll\left |\pi^2 q_n - p_n\right |\ll\frac{1}{q_n}.
\end{equation}
But, this is false for all sufficiently large numbers $n\geq1$. Hence, $\mu=\mu(\pi^2)=2$.
\end{proof}

The numerical data in Table \ref{t2000.500} demonstrates the accuracy of this result.

%SSSSSSSSSSSSSSSSSSSSSSSSSSSSSSSSSSSSSSSSSSSSSSSSSSSSSSSSSSSSSSSSSSSSSSSSSSSSSSSSSSSSSSSSSSSSSSSSSSSSSS
%SSSSSSSSSSSSSSSSSSSSSSSSSSSSSSSSSSSSSSSSSSSSSSSSSSSSSSSSSSSSSSSSSSSSSSSSSSSSSSSSSSSSSSSSSSSSSSSSSSSSSS
%SSSSSSSSSSSSSSSSSSSSSSSSSSSSSSSSSSSSSSSSSSSSSSSSSSSSSSSSSSSSSSSSSSSSSSSSSSSSSSSSSSSSSSSSSSSSSSSSSSSSSS
\section{The Irrational Numbers $\pi^r$ } \label{SP7777}
The first and second results explicate the arithmetic nature of the real number $\pi$.
\begin{thm} \label{thmP7777.100}  {\normalfont (Lambert)} The real number $\pi$ is irrational.
\end{thm}
\begin{proof} Confer \cite{NI1947} for the best known proof. 
\end{proof}

\begin{thm} \label{thmP7777.200}  {\normalfont (vonLindemann)} The real number $\pi$ is transcendental.
\end{thm}
\begin{proof} In \cite[p.\ 5]{BC2015}, there is a discusssion about the different proofs of this  result.
\end{proof}

Given the strong property of transcendence of the real number $\pi$, the irrationality of any rational power $\pi^r$ has a simple elementary proof, which is included for completeness.
\begin{thm} \label{thmP7777.300}  Let $r\ne0$ be an rational number. Then, the real number $\pi^r$ is irrational.
\end{thm}
 \begin{proof} Assume it is a rational number $\pi^{t/s}=a/b$, where $a,b,s,t \in \N^{\times}$ are fixed integers, and rewrite it as
 \begin{equation} \label{eqP7777.200}
\pi = \sqrt[t]{(a/b)^s}.
 \end{equation}
By Theorem \ref{thmP7777.200}, the rationality assumption is false, it contradicts the transcendental property of $\pi$. Hence, the real number $\pi^r\in \R$ is not a rational number. 
\end{proof}

%SSSSSSSSSSSSSSSSSSSSSSSSSSSSSSSSSSSSSSSSSSSSSSSSSSSSSSSSSSSSSSSSSSSSSSSSSSSSSSSSSSSSSSSSSSSSSSSSSSSSSS
%SSSSSSSSSSSSSSSSSSSSSSSSSSSSSSSSSSSSSSSSSSSSSSSSSSSSSSSSSSSSSSSSSSSSSSSSSSSSSSSSSSSSSSSSSSSSSSSSSSSSSS
%SSSSSSSSSSSSSSSSSSSSSSSSSSSSSSSSSSSSSSSSSSSSSSSSSSSSSSSSSSSSSSSSSSSSSSSSSSSSSSSSSSSSSSSSSSSSSSSSSSSSSS
\section{Numerical Data For The Exponent $\mu(\pi^2)$}\label{SP2000}
The continued fraction of the irrational number $\pi^2$ is
\begin{equation}\label{eqP2000.110}
\pi^2=[9; 1, 6, 1, 2, 47, 1, 8, 1, 1, 2, 2, 1, 1, 8, 3, 1, 10, 5, 1, 3, 1, 2, 1, 1, 3, 15, \ldots ].   
\end{equation}
The sequence of convergents $\{p_n/q_n: n \geq 1\}$ is computed via the recursive formula provided in the Lemma below. This result is standard results in the literature, see \cite{RD1996}, \cite{RH1994}, et alii. 

\begin{lem} \label{lemP2000.101} Let $\alpha=\left [ a_0, a_1, \ldots, a_n, \ldots, \right ]$ be the continue fraction of the real number $\alpha \in \R$. Then the following properties hold.
%\begin{multicols}{2}
\begin{enumerate} [font=\normalfont, label=(\roman*)]
\item$ \displaystyle  p_n=a_np_{n-1}+p_{n-2},$ \tabto{6cm} $p_{-2}=0, \quad p_{-1}=1$, \quad for all $n\geq 0.$
\item$ \displaystyle  q_n=a_nq_{n-1}+q_{n-2},$ \tabto{6cm} $q_{-2}=1, \quad q_{-1}=0$, \quad for all $n\geq 0.$
\item$ \displaystyle  p_nq_{n-1}-p_{n-1}q_{n}=(-1)^{n-1},$ \tabto{6cm} for all $n\geq 1.$
\item$ \displaystyle  \frac{p_n}{q_{n}}=a_0+\sum_{0 \leq k < n}\frac{(-1)^{k}}{q_kq_{k+1}},$ \tabto{6cm} for all $n\geq 1.$

\end{enumerate}
\end{lem}
%\end{multicols}

The $n$th convergent has a fast calculation algorithm, quite similar to the calculation of the $n$th Fibonacci number
\begin{equation}\label{eqP2000.020}
\begin{bmatrix}
F_{n+1} & F_{n}\\
F_{n}& F_{n-1}
\end{bmatrix}
=\begin{bmatrix}
1 & 0\\
1 & 1
\end{bmatrix} ^n,
\end{equation} 
which has a time complexity of $O(n(\log n)^c)$ arithmetic operations, for some constant $c\geq0$.

\begin{lem} \label{lemP2000.150} Let $\alpha=\left [ a_0, a_1, \ldots, a_n, \ldots, \right ]$ be the continue fraction of the real number $\alpha \in \R$. Then, the convergents are given by
\begin{equation}\label{eqP2000.040}
\frac{p_{n+1}}{q_{n+1}}=\begin{bmatrix}
a_0 & 0\\
1 & 1
\end{bmatrix} 
\begin{bmatrix}
a_1 & 0\\
1 & 1
\end{bmatrix} 
\begin{bmatrix}
a_2 & 0\\
1 & 1
\end{bmatrix}\cdots
\begin{bmatrix}
a_{n+1} & 0\\
1 & 1
\end{bmatrix} 
\begin{bmatrix}
p_n & p_{n-1}\\
q_n & q_{n-1}
\end{bmatrix}  
\end{equation}
\end{lem}

\begin{proof}Use induction to prove the matrix representation. 
\end{proof}
The approximation $\mu_n(\pi^2)$ of the irrationality measure in the inequality
\begin{equation} \label{eqP2000.050}
  \left | \pi^2-\frac{p_n}{q_n} \right | 
\geq\frac{1}{q^{\mu_n(\pi^2) }}.
\end{equation}
are tabulated in Table \ref{t2000.500} for the early stage of the sequence of convergents $p_n/q_n \longrightarrow \pi^2$. The values of the approximate irrationality measure $\mu_n(\alpha)\geq2$ of an irrational number $\alpha\ne0$ is defined by
\begin{equation}\label{eq2000.500}
\mu_n(\alpha)=-\frac{\log \left |\alpha-q_n/p_n\right | }{\log q_n}, 
\end{equation}
where $n\geq 2$. The range of values for $n\leq 30$ are plotted in Figure \ref{fi2000.500}.
%\newpage
\begin{figure}[h]
\caption{Approximate Irrationality Measure $\mu_n(\pi^2)$ of the number $\pi^2.$} \label{fi2000.500}\centering%
  \begin{tikzpicture}
	\begin{axis}[
		xlabel=$n$,
		ylabel=$\mu_n(\pi^2)$,
width=0.95\textwidth,
       height=0.5\textwidth		]
	\addplot[color=red,mark=square] coordinates {
		(3,2.253500)
		(4,2.511334)
		(5,3.236253)
		(6,2.018434)
		(7,2.321958)
		(8,2.064841)
		(9,2.090224)
		(10,2.107694)
		(11,2.098602)
		(12,2.071191)
		(13,2.049770)
		(14,2.172439)
		(15,2.094189)
		(16,2.021982)
		(17,2.147582)
		(18,2.095357)
		(19,2.018903)
		(20,2.074380)
		(21,2.023038)
		(22,2.055226)
		(23,2.032519)
		(24,2.031079)
		(25,2.054176)
		(26,2.110031)
		(27,2.020459)
		(28,2.030798)
		(29,2.036971)
		(30,2.039154)
		};
		\addplot[color=black,mark=square] coordinates {
		(2,2)
		(3,2)
		(4,2)
		(5,2)
		(6,2)
		(7,2)
		(8,2)
		(9,2)
		(10,2)
		(11,2)
		(12,2)
		(13,2)
		(14,2)
		(15,2)
		(16,2)
		(17,2)
		(18,2)
		(19,2)
		(20,2)
		(21,2)
		(22,2)
		(23,2)
		(24,2)
		(25,2)
		(26,2)
		(27,2)
		(28,2)
		(29,2)
		(30,2)
		
	};
	\end{axis}
\end{tikzpicture}	
\end{figure}
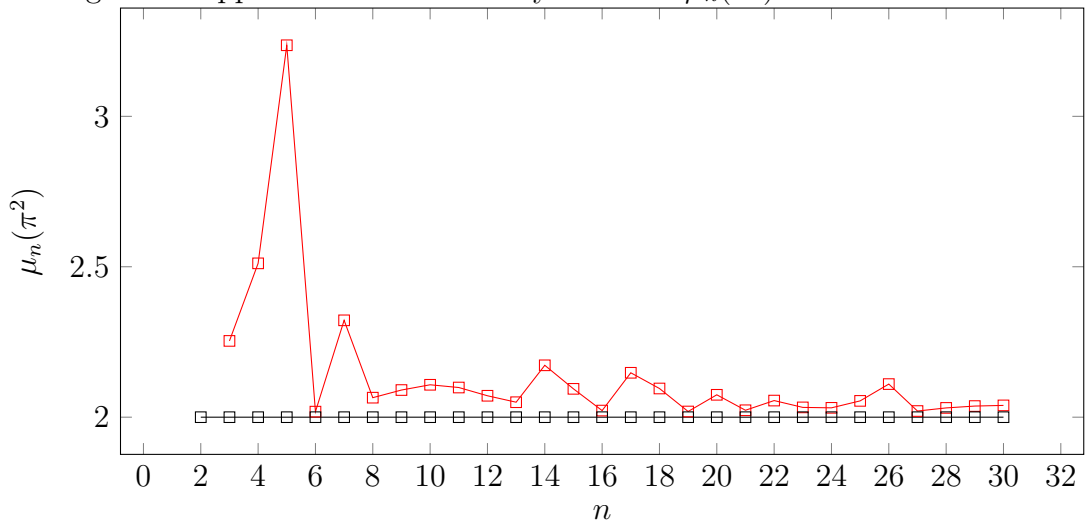

%PPPPPPPPPPPPPPPPPPPPPPPPPPPPPPPPPPPPPPPPPPPPPPPPPPPPPPPPPPPPPPPPPPPPPPPPPPPPPPPPPPPPPPPPPPPPPPPPPPPPPPPPPPPPPPPPPPPPPPPP
%PPPPPPPPPPPPPPPPPPPPPPPPPPPPPPPPPPPPPPPPPPPPPPPPPPPPPPPPPPPPPPPPPPPPPPPPPPPPPPPPPPPPPPPPPPPPPPPPPPPPPPPPPPPPPPPPPPPPPPPP
%PPPPPPPPPPPPPPPPPPPPPPPPPPPPPPPPPPPPPPPPPPPPPPPPPPPPPPPPPPPPPPPPPPPPPPPPPPPPPPPPPPPPPPPPPPPPPPPPPPPPPPPPPPPPPPPPPPPPPPPP
%PPPPPPPPPPPPPPPPPPPPPPPPPPPPPPPPPPPPPPPPPPPPPPPPPPPPPPPPPPPPPPPPPPPPPPPPPPPPPPPPPPPPPPPPPPPPPPPPPPPPPPPPPPPPPPPPPPPPPPPP
\section{Open Problems}\label{exe9090}

\begin{exe}\label{exe9090.123} {\normalfont Given the partial quotients $\alpha=[a_0;a_1,a_2,a_3,\ldots]$, develop a fast algorithm, based on matrix multiplication, for computing the $n$th convergent $p_n/q_n$. 
}
\end{exe}
\begin{exe}\label{exe9090.125} {\normalfont Given the partial quotients $\alpha=[a_0;a_1,a_2,a_3,\ldots]$, prove that the $n$th convergent $p_n/q_n$ of a quadratic irrational number $\alpha\ne0$ has polynomial time complexity, $O(n(\log n)^c)$ arithmetic operations, for some constant $c\geq0$.
}
\end{exe}

%BBBBBBBBBBBBBBBBBBBBBBBBBBBBBBBBBBBBBBBBBBBBBBBBBBBBBBBBBBBBBBBBBBBBBBBBBBBBBBBBBBBBBBBBBBBBBBBBBBBBBBBBB
%BBBBBBBBBBBBBBBBBBBBBBBBBBBBBBBBBBBBBBBBBBBBBBBBBBBBBBBBBBBBBBBBBBBBBBBBBBBBBBBBBBBBBBBBBBBBBBBBBBBBBBBBB
%BBBBBBBBBBBBBBBBBBBBBBBBBBBBBBBBBBBBBBBBBBBBBBBBBBBBBBBBBBBBBBBBBBBBBBBBBBBBBBBBBBBBBBBBBBBBBBBBBBBBBBBBB
%\newpage

%\begin{wraptable}
\begin{table}[h!]
\centering
\caption{Numerical Data For The Exponent $\mu(\pi^2)$ And Lagrange Number For $\pi^2$} \label{t2000.500}
\begin{tabular}{l|l|l| l|r}
$n$&$p_n$&$q_n$&$\mu_n(\pi^2)$&$q_n^{\mu_n(\pi^2)-2}$\\
\hline
1&$9$&   $1$   &$ $&$1.000000$\\
2&$10$&  $1$   &$ $&$1.000000$\\
3&$69$&   $7$   &$2.253500$&$1.637692$\\
4&$79$&  $8$   &$2.511334$&$2.895880$\\
5&$227$&   $23$   &$3.236253$&$48.243646$\\
6&$10748$&  $1089$   &$2.018434$&$1.137587$\\
7&$10975$&  $1112$   &$2.321958$&$9.565723$\\
8&$98548$&   $9985$   &$2.064841$&$1.816861$\\
9&$109523$&  $11097$   &$2.090224$&$2.317259$\\
10&$208071$&  $21082$   &$2.107694$&$2.921860$\\ 
11&$525665$&      $53261$   &$2.098602$&$2.924377$ \\ 
12&$1259401$&   $127604$   &$2.071191$&$2.309360$\\
13&$1785066$&   $180865$   &$2.049770$&$1.826663$\\
14&$3044467$& $308469$ & $2.172439$&$8.842074$ \\
15&$26140802$&  $2648617$ & $2.094189$ &$4.026964$\\ 
16&$81466873$&  $8254320$ & $2.021982$&$1.419196$ \\
17&$107607675$&    $10902937$ & $2.147582$ &$10.929864$\\
18&$1157543623$&  $117283690$ & $2.095357$ &$5.881096$\\
19&$5895325790$&  $597321387$ & $2.018903$ &$1.465199$\\
20&$7052869413$&  $714605077$ & $2.074380$ &$4.555808$\\
21&$27053934029$&  $2741136618$ & $2.023038$&$1.649799$ \\
22&$34106803442$&  $3455741695$ & $2.055226$&$3.363376$ \\
23&$95267540913$&  $9652620008$ & $2.032519$ &$2.111984$\\
24&$129374344355$&$13108361703$&$2.031079$&$2.062734$\\
25&$224641885268$&  $22760981711$& $2.054176$&$3.640082$\\
26&$803300000159$&   $81391306836$   &$2.110031 $&$15.867255$\\
27&$12274141887653$&  $1243630584251$   &$2.020459 $&$1.767850$\\
28&$13077441887812$&   $1325021891087$   &$2.030798$&$2.362328$\\
29&$25351583775465$&  $2568652475338$   &$2.036971$&$2.876068$\\
30&$63780609438742$&  $6462326841763$ & $2.039154$&$3.173788$
\end{tabular}
\end{table}
%\end{wraptable}

%\currfilename.\\
%\color{blue}

\end{document}